\documentclass[12pt]{amsart}
 \usepackage[latin1]{inputenc}      
\usepackage{tikz}
\usetikzlibrary{snakes}
\usetikzlibrary{arrows}

\usepackage{amssymb,verbatim}
\usepackage{epsfig}

\newtheorem{theorem}{Theorem}[section]
\newtheorem{corollary}[theorem]{Corollary}
\newtheorem{lemma}[theorem]{Lemma}

\newtheorem{proposition}[theorem]{Proposition}

\theoremstyle{definition}
\newtheorem{definition}[theorem]{Definition}
\newtheorem{remark}[theorem]{Remark}

\newtheorem{assumption}[theorem]{Assumption}

\newtheorem{example}[theorem]{Example}

\newenvironment{proofof}[1]{\noindent {\bf Proof of #1.}}{ \hfill\qed\\ }

\newcommand{\To}{\rightarrow}

\newcommand{\orb}{\text{orb}}
\newcommand{\proj}{\text{proj}}

\newcommand{\card}{\text{card}}

\newcommand{\dist}{\mbox{dist}}

\newcommand{\bbr}{\mbox{$\mathbb{R}$}}
\newcommand{\bbz}{\mbox{$\mathbb{Z}$}}

\newcommand{\bbn}{\mbox{$\mathbb{N}$}}
\newcommand{\bbp}{\mbox{$\mathbb{P}$}}

\newcommand{\di}{\text{d}}
\def\S{\mathbb{S}^1}
\def\SP{\mathbb{S}^2}

\newcommand{\dP}{\partial P}

\begin{document}

\title[Homotopical rigidity of polygonal billiards]{Homotopical rigidity of polygonal billiards}

\author{Jozef Bobok}

\author{Serge Troubetzkoy}

\address{Department of Mathematics of FCE\\Czech Technical University in Prague\\
Th\'akurova 7, 166 29 Prague 6, Czech Republic}
\email{bobok@mat.fsv.cvut.cz}

\address{Aix-Marseille Universit\'e\\
Centre de physique th\'eorique\\
Federation de Recherches des Unites de Mathematique de Marseille\\
Institut de math\'ematiques de Luminy\\
Luminy, Case 907, F-13288 Marseille Cedex 9, France}
\email{troubetz@iml.univ-mrs.fr}
\urladdr{http://iml.univ-mrs.fr/{\lower.7ex\hbox{\~{}}}troubetz/} \date{}

\begin{abstract}
Consider two $k$-gons $P$ and $Q$.  We say that the billiard flows in $P$ and $Q$ are homotopically equivalent
if the set of conjugacy classes  in the fundamental group of $P$ which contain a periodic billiard orbit
agrees with the analogous set for $Q$.  We study this equivalence relationship and compare it to the equivalence
relations, order equivalence and code equivalence, introduced in \cite{BT1,BT2}.
In particular we show if $P$ is a rational polygon, and $Q$ is homotopically equivalent to $P$, then
$P$ and $Q$ are similar, or affinely similar if all sides of $P$ are vertical and horizontal.
\end{abstract}\maketitle

\vspace{-0.5cm}
\section{Introduction}
In mathematics one often wants to know if one can reconstruct an
object (often a geometric object) from certain discrete
data, i.e., does rigidity hold.  A well known example is a question posed by Burns and Katok
whether  a negatively
curved surface is determined by its marked length spectrum \cite{BK}.
The ``marked length spectrum'' of a surface $S$ is the function that
associates to each conjugacy class in the fundamental group $\pi_1(S)$ the length of the
geodesic in the class. This question
was resolved positively by Otal \cite{O} and Croke \cite{C}.  See \cite{S} for a nice survey of
rigidity theory.

In this article we consider the same question for polygonal billiard tables (see \cite{MT} for details about polygonal billiards).
More precisely,
fix $P$ a simply connected polygon with $k$-vertices. The billiard in $P$ is
describe by a point mass which moves with unit speed without
friction in the interior of $P$ and reflects elastically from the boundary.
We think of $P$ as having a top and a bottom, when the billiard ball
arrives at the boundary it continues on the other side of $P$. This makes $P$ with the corners removed
into a  punctured sphere with $k$-punctures on the equator $E$ (consisting of the sides of $P$ without the punctures) and a flat metric away from these punctures.
This is often called the pillow case model,
we will denote the punctured sphere by $\bbp$ or more precisely by $(\bbp,E,\{p_1,p_2,\dots,p_k\})$ where $p_i$ denote the removed points and the flat metric is implicit.
We will used the phrase ``arc-side'' to
denote a side of $P$  in the pillowcase model.
The billiard flow $\{T^P_t\}_{t\in\tiny\bbr}$ on $\bbp$ is then
the geodesic flow on the unit tangent bundle $T\bbp$ of this sphere.

  Let $D(P)$ be the set of conjugacy classes
in $\pi_1(\bbp)$ which contain a closed billiard trajectory, the set $D(P)$ is the domain of the marked length spectrum map. Consider a polygon $Q$ with the same number of sides as $P$
and call $Q$ homotopically equivalent to $P$, denoted $Q \equiv_{\text{hom}} P$, if $D(Q) = D(P)$.
We show that if $Q$ is a polygon homotopically equivalent to a rational polygon $P$ (all angles between sides are rational multiples of $\pi$) then
$Q$ is similar to $P$ (or affinely similar if all sides of $P$ are vertical or horizontal) (Corollary \ref{main}).
Thus among the rational billiard metrics on the $k$ punctured sphere the domain of the marked length spectrum
determines the polygon up to similarity (resp.\ affine similarity).

We study the set of free homotopy classes in the $k$ punctured sphere.  Motivated by the billiard case we assign
symbolic codes to each free homotopy class  and show that they these codes are in bijection with the free homotopy classes (Theorem \ref{biject}, Corollary \ref{c:1}, Proposition \ref{p:2}). In Theorem  \ref{compare} we compare the equivalence
relation defined above with the ones defined in \cite{BT1} and \cite{BT2}, order equivalence and code equivalence.
Our main result then follows from the bijection between homotopy classes and codes
established in Corollary \ref{c:1}, and using the techniques  and results of \cite{BT2} and \cite{BT1}.

\section{Homotopy equivalences of closed curves}

For $k\in \bbn$, let a pillowcase $(\bbp,E,\{p_1,p_2,\dots,p_k\})$ be the sphere with the standard topology
and $E$ be a simple closed curve in the sphere with punctures $p_i$ corresponding to parameters  in $\S$ ordered counterclockwisely. We need to identify pillowcases ($\mathbb{P},E,\{p_1,p_2,\dots,p_k\})$  and $(\mathbb{P}',E',\{p'_1,p'_2,\dots,p'_k\})$ with same number of punctures. This is possible owing to the Jordan-Schoenflies Theorem, a sharpening of the Jordan curve theorem.

\begin{theorem}\cite{M77}~\label{JS}A simple closed curve in the plane can be mapped onto a circle by a homeomorphism of the whole of $\bbr^2$.
\end{theorem}

Since any homeomorphism of the plane can be extended to a homeomorphism of the Riemann sphere $\SP$ that is homeomorphic to $\overline{\bbp}=\bbp\cup\{p_1,p_2,\dots,p_k\}$, Theorem \ref{JS} implies the following fact.

\begin{proposition}\label{p:3}For ($\mathbb{P},E,\{p_1,p_2,\dots,p_k\})$  and $(\mathbb{P}',E',\{p'_1,p'_2,\dots,p'_k\})$ there is a homeomorphism $\phi\colon\overline{\bbp}\to\overline{\bbp'}$ such that $\phi(E) = E'$ and $\phi(p_i) = p'_i$ for each $i$.\end{proposition}
\begin{proof}By Theorem \ref{JS} we can consider a homeomorphism $\psi\colon\overline{\bbp}\to\SP$, resp.\ $\psi'\colon\overline{\bbp'}\to\SP$  such that $\psi(\overline{E}) = \S$, resp.\  $\psi'(\overline{E'}) = \S$. Since the punctures $\{p_1,p_2,\dots,p_k\}$, resp.\ $\{p'_1,p'_2,\dots,p'_k\}$ are parametrized counterclockwisely, there also exists a homeomorphism $\tau$ of $\SP$ such that $\tau(\S)=\S$ and $\tau(\{p_i\})=\psi'(\{p'_i\})$ for all $i$. Clearly the map $\phi=(\psi')^{-1}\circ\tau\circ\psi$ is a homeomorphism with the required properties.\end{proof}

A {\it pillowcase code} $\sigma$  is a finite sequence $\sigma = (\sigma_0 \sigma_1 \cdots \sigma_{2n-1})$ with $\sigma_i \in \{1,2,\dots,k\}$ of even length
such that $\sigma_i \not = \sigma_{i+1}$ (througout this section pillowcase code subscripts will be taken modulo the length of the code, here $2n$).
We call two pillowcase codes $\sigma = (\sigma_0 \sigma_1 \cdots \sigma_{2n-1})$ and  $\sigma'= (\sigma'_0 \sigma'_1 \cdots \sigma'_{2m-1})$ equivalent if
$m=n$ and there is a $j$ so that $\sigma_{i+j} = \sigma'_i$ for all $i$.
It is easy to see that this is an equivalence relation which we denote $\sigma \equiv \sigma'$.

The main result of this section is the following theorem.

\begin{theorem}\label{biject} Fix $k \ge 3$, and a pillowcase $(\bbp,E,\{p_1,p_2,\dots,p_k\})$.
There is a bijection between the set of free homotopy classes of closed curves on the $k$-punctured sphere
and  the set of equivalence classes of pillowcase codes.
\end{theorem}

To  prove the theorem we need a series of lemmas.
\begin{lemma}\label{l:2} Fix $\{q_1,q_2,\dots,q_m\} \subset  \bbr^2$ and let $S=\bbr^2\setminus\{q_1,q_2,\dots,q_m\}$ be equipped with the standard topology, i.e., the restriction of the topology of
$\bbr^2$. Assume that $\gamma\colon\S\to S$ is a  homotopically trivial curve in $S$. Then $\{q_1,q_2\dots,q_m\}
\subset  U_{\gamma}$, where $U_{\gamma}$ is the unique unbounded component of $S\setminus \gamma$.  \end{lemma}

\begin{proof}Assume that $q_i \notin U_{\gamma}$ for some $i$. For a point $p\in S$ consider a homotopy $H=H(s,t)\colon~[0,1]\times\S\to S$ such that $H(0,t)=\gamma(t)$ and $H(1,t)=p$. If we define a value $s_0$ as
$$s_0=\sup\{s\in [0,1]\colon~q_i \notin U_{H(s,\cdot)}\},$$
then clearly $0<s_0<1$ and $q_i=H(s_0,t_0)$ for some $t_0\in \S$, a contradiction.
\end{proof}

\begin{lemma}\label{l:3}Let $\gamma\colon\S\to \bbp$ be a closed homotopically trivial curve in the punctured sphere $\bbp$. Then there is a simply connected open subset $\Gamma\subset \bbp$ such that $\Gamma$ contains all punctures and $\gamma\cap\overline \Gamma=\emptyset$. \end{lemma}

\begin{proof}It follows from  Lemma \ref{l:2} that all punctures have to belong to the same component $G$ of $\bbp\setminus\gamma$. Since this component is open and connected, it is also arc-wise connected and there is a simple arc $A\subset G$ which is the closure of finitely many segments joining the punctures. Now, the set $\Gamma$ can be taken as an open $\delta$ neighborhood of $A$ for sufficiently small positive $\delta$.\end{proof}

A closed parametrized curve $\gamma$ in $\bbp$ is called nice if it intersects the equator a finite number of times at times (in the clockwise cyclic sense)
\begin{equation}\label{e:2}t(0) < t(1) \cdots < t(2m) = t(0),\end{equation}
two consecutive intersections are not in the same arc-side and the segments
$\{\gamma(t) : t \in [t_i,t_{i+1}]\}$ alternate between the top and bottom of the sphere as we vary $i$.
Note that the set of times $t(i)$ are fixed, thus an origin of the curve is marked to be $\gamma(t(0))$,  a cyclic permutation of the $t(i)$ will give ``another'' nice curve.

Let $\proj_1 : \bbp \times \S \to \bbp$ denote the first natural projection (to the foot point). Notice that if $\hat{\gamma_u} = \{T_tu: t \in \mathbb{R}\}$ is a closed billiard orbit in $T\bbp$ such that $\proj_1(u) \in E$, then $\gamma_u = \proj_1(\hat{\gamma})$ is
 a nice curve.

The symbolic pillowcase code of a nice closed curve $\gamma$ in $\bbp$ is
the sequence
$\sigma(\gamma)=(\sigma_0(\gamma),  \sigma_1(\gamma), \cdots , \sigma_{2m-1}(\gamma)) \in \{1,\dots,k\}^{2m}$
given by
\begin{equation*}\gamma({t(i))}\in e^{\circ}(\sigma_i).\end{equation*}
In the same way that nice curves are defined up to a cylic permutation, symbolic pillowcase codes
are also only defined up to a cyclic permutation of the $t(i)$'s, thus we will
say $\sigma(\gamma_1) \equiv \sigma(\gamma_2)$ if there is a $j$ so that $\sigma(\gamma_1)_{i+j } = \sigma(\gamma_2)_i$
for all $i$.

Let $[\alpha,\beta] \subset E$ denote the  oriented segment starting at $\alpha$ and ending at $\beta$.
For parameter intervals in $\S$ we will use the notation $[t_0,t_1]$ for the clockwise oriented interval and $[t_0,t_1]^{cc}$ for the counterclockwise
case.


\vspace{1.5pt}
\begin{figure}
\begin{center}\hskip-19mm\includegraphics[width=14.5cm]{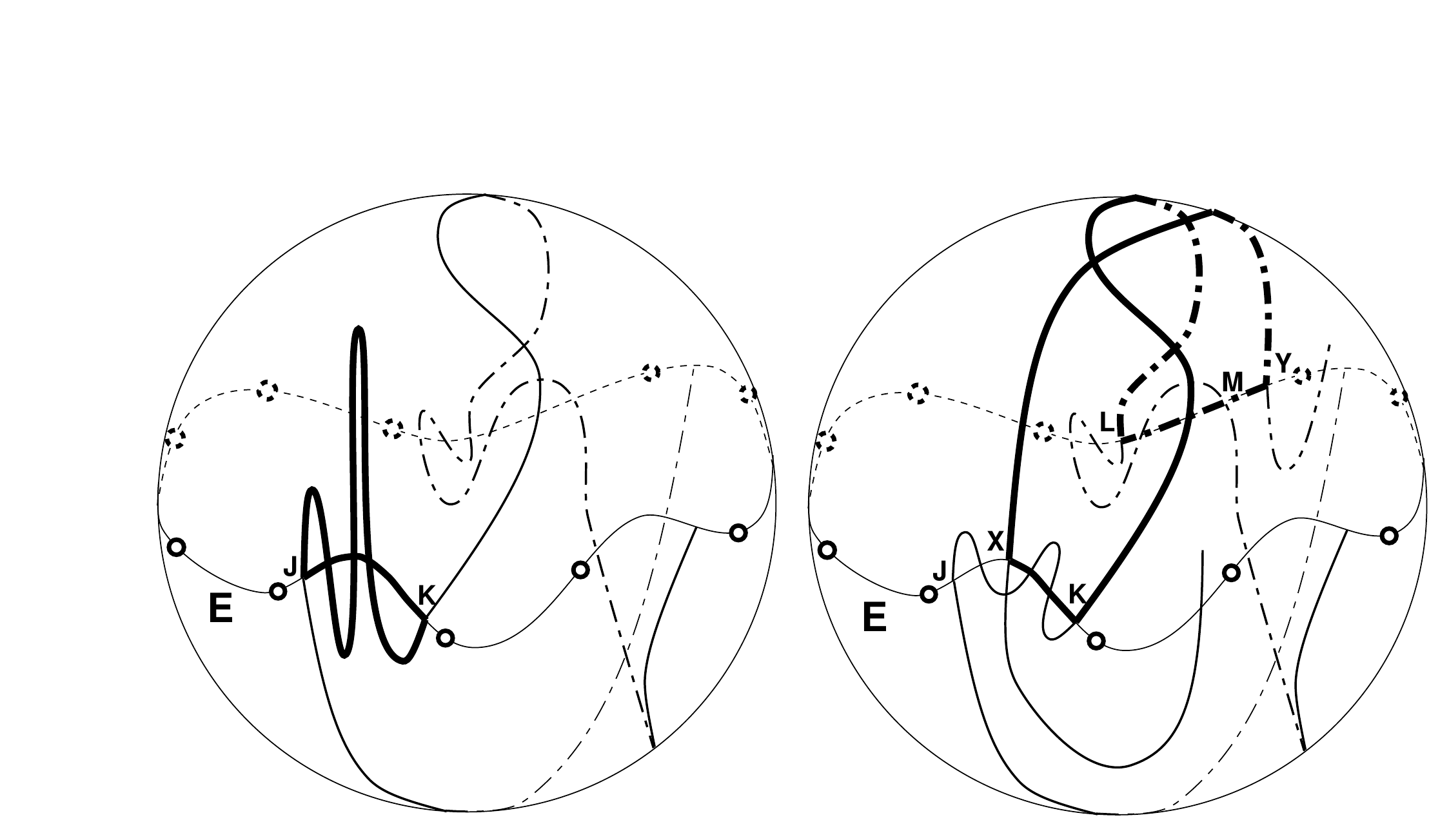}\end{center}
\caption{Definition \ref{def}(iii)-left and (iv)-right. $J=\delta(t_i^-)$, $K=\delta(t_i^+)$, $L=\delta(t_{i+1}^-)$, $M=\delta(t_{i+1}^+)$; $X=\gamma(t(i))$, $Y=\gamma(t(i+1))$.}
\end{figure}
\vskip2mm

\begin{definition}\label{def}Let $\gamma $ be a nice curve in $\bbp$
 with symbolic pillowcase code  $\sigma(u)=(\sigma_0\sigma_1\cdots\sigma_{2m-1})$. We say that a closed curve $\delta=\delta(t)$, $t\in\S$, on the sphere $\bbp$ is $\gamma$-admissible if there are $t^-_0,t^+_0,\dots,t^-_{2m},t^+_{2m}\in\S$ such that \end{definition}
\begin{itemize}
\item[(i)] $t^-_0\le t^+_0<t^-_1\le t^+_1 < \cdots<t^-_{2m} = t^-_0\le t^+_{2m} = t^+_0$ (in clockwise cyclic sense),
\item[(ii)] for each $0\le i\le 2m-1$, $\delta(t^-_i),\delta(t^+_i)\in e^{\circ}(\sigma_i)$,
\item[(iii)] for each $0\le i\le 2m-1$,  the set $$\Big \{\delta(t)\colon~t\in [t^-_i,t^+_i] \Big \}\cup [\delta(t^+_i),\delta(t^-_i)]$$ is a homotopically trivial curve,
\item[(iv)] for each $0\le i\le 2m-1$, the set
$$\{\gamma(t)\}_{[t(i),t(i+1)]}\cup [\gamma(t(i+1)),\delta(t^-_{i+1})]\cup
\{\delta(t)\}_{[t^-_{i+1},t^{+}_{i}]^{cc}}\cup
[\delta(t^+_i),\gamma(t(i))]$$
is a homotopically trivial curve.
\end{itemize}

\begin{lemma}\label{l:1}Let $\gamma$ be a nice closed curve in $\bbp$.  Assume that a closed curve $\delta=\delta(t)$, $t\in\S$, on the sphere $\bbp$ is $\gamma$-admissible. Then every closed curve on the sphere $\bbp$ which is sufficiently close with respect to the Hausdorff metric to $\delta$ is also $\gamma$-admissible.\end{lemma}

\begin{proof} For each homotopically trivial curve $\beta\subset\bbp$ given by  Definition \ref{def}(iii)-(iv) applied to the curve $\delta$ we use  Lemma \ref{l:3}
to produce a simply connected open subset $\Gamma_{\beta}$ of $\bbp$.  Since there are finitely many such $\beta$'s and $\beta\cap \overline{\Gamma_{\beta}}=\emptyset$, we get
\begin{equation*}\eta=\min_{\beta}\dist(\Gamma_{\beta},\beta)>0.\end{equation*}
            By continuity for any closed curve less than $\eta$ close with respect to the Hausdorff metric to
the closed curve  $\delta$,  there exist parameters $t_i^{\pm}$ such that the requirements (i)-(ii) of Definition \ref{def} are fulfilled and the curves defined in (iii)-(iv)  are homotopically trivial. \end{proof}

\begin{proposition}\label{code-homotop} Consider two nice closed curves  $\gamma_1,\gamma_2$ in $\bbp$.
Then $\gamma_1$ and $\gamma_2$ are from the same free homotopy class if and only if $\sigma(\gamma_1) \equiv \sigma(\gamma_2)$.
\end{proposition}
\begin{proof}

Let $H=H(s,t)$, $(s,t)\in [0,1]\times\S$ be a homotopy between $\gamma_1$ and $\gamma_2$ such that $H(0,t)=\gamma_1(t)$, $H(1,t)=\gamma_2(t)$, $t\in\S$.
Let
$$S=\{s\in [0,1]\colon~\{H(s,t)\}_{t\in\S}\text{ is $\gamma_1$-admissible}\} \
\text{and}\ \bar s=\sup S.$$
First note that by the definition of symbolic code $\gamma_1$ is $\gamma_1$-admissible with times $t_i^- = t_i^+ = t(i)$.
It follows from Lemma \ref{l:1}  that $\bar{s}>0$. Now, take an increasing sequence $\{s_n\} \subset [0,1]$ such that
$\lim_ns_n=\bar{s}$ and for each $i\in\{0,\dots,2m\}$ the limits $\lim_nt^-_{i}(s_n)=t^-_{i}(\bar{s})$ and $\lim_nt^+_{i}(s_n)=t^+_{i}(\bar{s})$ exist, where
$t^-_{i}(s_n)$, $t^+_{i}(s_n)$ satisfy Definition \ref{def}(i) for an admissible curve $\{H(s_n,t)\}_{t\in\S}$.
Then by continuity of $H$ we have
$$t^-_0(\bar{s})\le t^+_0(\bar{s}) \le t^-_1(\bar{s}) \le \cdots \le t^-_{2m}(\bar{s})=t^-_0(\bar{s}).$$
If $t_i^+(\bar{s}) = t_{i+1}^-(\bar{s})$ for some $i$ then $H(\bar{s},t_i^+(\bar{s})) = H(\bar{s},t_{i+1}^-(\bar{s}))$, which can not
happen since then this point would be in different arc-sides which is impossible.
Thus
$$t^-_0(\bar{s})\le t^+_0(\bar{s})< t^-_1(\bar{s}) \le t^+_1(\bar{s}) < \cdots<t^-_{2m}(\bar{s})=t^-_0(\bar{s})$$
again, i.e., the requirement of Definition \ref{def}(i) is fulfilled.

Moreover, since the homotopy $H$ is uniformly continuous on the space $[0,1]\times \S$ and the uniform limit of a homotopically trivial curve is homotopically trivial, the conditions (ii)-(iv) of Definition \ref{def} hold again. It shows that $\bar s\in S$. Thus we can apply Lemma \ref{l:1} to obtain $\bar s=1$. But it means that also $\gamma_2$ is $\gamma_1$-admissible.

We have shown that the code $\sigma(\gamma_1)$ is seen in the nice curve $\gamma_2(t)$ at the times $t_i^+(\bar{s})$.  The code
of $\sigma(\gamma_2)$ might have a different choice of time 0, thus we have shown that
the code $(\sigma_0(\gamma_1) \sigma_1(\gamma_1)\cdots\sigma_{2m-1}(\gamma_1) )$ is contained in some cyclic permutation of the code
$(\sigma_0(\gamma_2) \sigma_1(\gamma_2)\cdots\sigma_{2m-1}(\gamma_2) )$  in an increasing order.
By symmetry, $\sigma(\gamma_1) \equiv\sigma(\gamma_2)$.

We assume that $\sigma(\gamma_1) \equiv\sigma(\gamma_2)$. Since a reparametrized closed curve is homotopic to the original curve, we can assume that the curves $\gamma_1$ and $\gamma_2$ are parametrized  such that $\sigma(\gamma_1)=\sigma(\gamma_2)$ and that the times defined in (\ref{e:2}) are the same.
Let us denote $\langle\gamma_1(t),\gamma_2(t)\rangle$ the subarc in $\bbp$ of the geodesic and lying in the same hemisphere as its endpoints $\gamma_1(t),\gamma_2(t)$.
We can define a ``geodesic'' homotopy $H(s,t)$, $(s,t)\in [0,1]\times \S$, from $\gamma_1$ to $\gamma_2$ by letting $H(s,t)$ the point of $\langle\gamma_1(t),\gamma_2(t)\rangle$ satisfying $$(1-s)~\dist(\gamma_1(t),H(s,t))=s~\dist(\gamma_2(t),H(s,t)).$$\end{proof}

Next we show that this theorem is of interest for all closed curves.

\begin{proposition}\label{nicerep}
Let $\gamma_1$ be a closed curve which is not homotopically trivial. Then there exists a nice closed curve $\gamma_2$ which is homotopic to $\gamma_1$.
\end{proposition}

\begin{proof}
Since $\gamma_1$ is compact, there exists on open neighborhood $O$ of the punctures $\{p_1,p_2,\cdots,p_k\}$
such that $\gamma_1 \cap \overline{O} = \emptyset$.  Thus if $\gamma_1(t_1)$ and $\gamma_1(t_2)$
are in different arc-sides, then there is some non-degenerate interval $I  \subset [t_1,t_2]$ such that $\gamma_1(t)$ is not in an arc-side
for all $t \in I^{\circ}$.  We can find a sequence of times $t_0^- \le t_0^+ < t_1^- \le t_1^+ < t_2^- \cdots$, which is either finite, or
$\lim_{m \to \infty} t^{\pm}_m = t_{\infty}$ for some $t_\infty\in\S$, such that for each $i$
\begin{itemize}
\item there is a side $e(j_i)$ such that
all visits of $\gamma_1(t)$ of a side for $t \in [t_i^-,t_i^+]$  are in $e(j_i)$ and
\item $\gamma_1(t)$ is not in an arc-side for all
$t \in (t_i^+,t_{i+1}^-)$.
\end{itemize}
Furthermore by continuity $\gamma(t_{\infty})$ is in at least two sides,
which is not possible. Thus  we must have only a finite number of $t_i$'s.

We will now construct the nice curve by modifying $\gamma_1$ in a homotopic way.
Fix $i$. There are two cases,  either there exists small positive $\varepsilon$ so that   $\gamma_1(t_i^- - t)$ and $\gamma_1(t_i^+ + t)$  are in the same hemisphere, or opposite hemispheres for all  $t \in (0,\varepsilon)$. The latter case has to happen at least twice since the curve $\gamma_1$ is not homotopically trivial.

Fix $t \in (0,\varepsilon)$. In the same hemisphere case we replace  the piece of the curve $\gamma([t_i^--t,t_i^++t])$ by a curve completely contained in the same open
hemisphere starting at $\gamma([t_i^--t])$  and ending at $\gamma([t_i^++t])$. In the case when they are in opposite hemispheres we
replace  the piece of the curve $\gamma([t_i^--t,t_i^++t])$ by a curve  starting at $\gamma([t_i^--t])$  and ending at $\gamma([t_i^++t])$
which crosses the equator exactly once in the arc-side $e(j_i)$.
\end{proof}

\begin{proofof}{Theorem \ref{biject}}
The theorem follows immediately by combining Propositions \ref{code-homotop} and \ref{nicerep}.
\end{proofof}

\subsection{Application to polygonal billiards}\label{application}

We have seen in Proposition \ref{p:3} that two pillowcases $$(\mathbb{P},E,\{p_1,p_2,\dots,p_k\}),~(\mathbb{P}',,E',\{p'_1,p'_2,\dots,p'_k\})$$  can be homeomorphically identified by a homeomorphism $\phi: \overline{\mathbb{P}}\to \overline{\mathbb{P}'}$ so that $\phi(E) = E'$ and $\phi(p_i) = p'_i$ for all $i$. This is what we call a homeomorphism between pillowcases.
Note that the homeomorphism does not preserve the flat metric.

Consider now two $k$-gons $P$ and $Q$.
We fix a cyclic labeling of the sides of $P$ and of $Q$, thus by the above we can identify their pillowcase models.  Up to cyclic permutation,
there are $k$ different identifications possible.
Each of these labelings can be identified with $\bbp$.
Fix a labeling of $Q$, thus are thinking of $\bbp$ as representing $P$ and $Q$ at the same time.

Since the projection of a closed billiard trajectory is a nice curve we have the following corollary.
\begin{corollary}\label{c:1} Let $P$ and $Q$ be homeomorphically identified with fixed labeling.
Let $\hat{\gamma_P}$, resp.\ $\hat{\gamma_Q}$ be a closed billiard trajectory on $P$, resp.\ $Q$.
Then $\gamma_P=\proj_1(\hat{\gamma_P})$ and $\gamma_Q=\proj_1(\hat{\gamma_Q})$ are from the same homotopy class
if and only if $\sigma(\gamma_P) \equiv \sigma(\gamma_Q)$.
\end{corollary}

\section{Polygonal billiards}

A polygonal billiard table is a polygon $P$.  Our polygons are assumed to be planar, simply connected, not necesarily convex, and compact, with all angles non trivial, i.e., in $(0,2\pi)\setminus \{\pi\}$.
The billiard
flow $\{T_t\}_{t\in\bbr}$ in $P$ is generated by the free motion of
a point mass subject to elastic reflections in the boundary. This
means that the point moves along a straight line in $P$ with a
constant speed until it hits the boundary. At a smooth boundary
point the billiard ball reflects according to the well known law of
geometrical optics: the angle of incidence equals  the angle of
reflection. If the billiard ball hits a corner, (a non-smooth
boundary point), its further motion is not defined. Additionally to
corners, the billiard trajectory is not defined for orbits
tangent to a side.

By $D$ we denote
the group generated by the reflections in the lines through the
origin, parallel to the sides of the polygon $P$. The group $D$ is
either

\begin{itemize}\item finite, when  all the angles of $P$ are of the form $\pi
m_i/n_i$ with distinct co-prime integers $m_i$, $n_i$,  in this
case $D=D_N$ the dihedral group generated by the reflections in
lines through the origin that meet at angles $\pi/N$, where $N$
is the least common multiple of $n_i$'s,\end{itemize} or
\begin{itemize}\item countably infinite, when at least one angle
between sides of $P$ is an irrational multiple of $\pi$.
\end{itemize} In the two cases we will refer to the polygon as
rational, respectively irrational.

Consider the phase space $P\times \S$ of the billiard flow $T_t$,
and for $\theta\in \S$, let $R_{\theta}$ be its subset of points
whose second coordinate belongs to the orbit of $\theta$ under $D$.
Since a trajectory changes its direction by an element of $D$ under
each reflection, $R_{\theta}$ is an invariant set of the billiard
flow $T_t$ in $P$.
In fact if $P$ is a rational polygon then the set $R_{\theta}$ is a compact surface
with conical singularities and a flat metric away from the singularities, i.e., a  translation
surface. The set $P\times \theta$ will be called a
floor of the phase space of the flow $T_t$.

The billiard map $T\colon~V_P=\cup e\times\Theta\subset \dP\times
(-\frac{\pi}{2},\frac{\pi}{2})\To V_P$ associated with the flow $T_t$ is the first return map
to the boundary $\dP$ of $P$. Here the union $\cup e \times \Theta$
is taken over all sides of $P$ and for each side $e$ over the inner
pointing directions $\theta\in\Theta = (-\frac{\pi}{2},\frac{\pi}{2})$ measured with respect to the inner pointing normal.  We will denote points of $V_P$ by $u = (x,\theta)$. We sometimes use the map $\varrho\colon~V_P\to \bbr^+$ defined as
\begin{equation}\label{rho}\varrho(u,\tilde u)=\max\{\vert\proj_1(u)-\proj_1(\tilde u)\vert, \vert\proj_2(u)-\proj_2(\tilde u)\vert\}.\end{equation}
As usual, $\proj_1$, resp.\  $\proj_2$
denotes the first natural projection (to the foot
point), resp.\  the second natural projection (to the
direction). Clearly the map $\varrho$ is a metric.

The billiard map $T$ has a
natural invariant measure on its phase space $V_P$, the phase area given by the formula
$\sin\theta~\di x~\di\theta$.  In the case, when $P$ is rational, this measure
is not ergodic since for each $\theta$ the skeleton $K_P =V_P \cap R_{\theta}$ of the
surface $R_{\theta}$ is an invariant set.
We will call the associated measure $\mu$ suppressing the $\theta$ dependance.
In particular, an edge $e$
of $V_P \cap R_{\theta}$ associated with $\theta\in (0,\pi)$ has the $\mu$-length $\vert
e\vert\cdot\sin\theta$.

A {\em saddle connection} is a  segment of a billiard orbit starting in a corner of $P$ and ending in
a corner of $P$ which does not pass through any corners in between.
We can concatenate a saddle connection which ends at a certain vertex
of $P$ with another saddle connection which starts at this vertex, if we arrive back at
the starting vertex after a finite number of such concatenations we call this a {\em saddle loop}.
Note that a saddle loop can consist of a single saddle connection.  Closed billiard trajectories always
appear in families; any such family fills a {\em maximal cylinder} bounded on each side by a saddle loop.
We will call such saddle loops {\em regular} and those which do not bound a closed cylinder {\em irregular}.

%

 A direction, resp.\  a point $u$ from the phase space is
exceptional if it is the direction of a saddle connection, resp.\
$\proj_2(u)$ is such a direction. Obviously there are countably many
saddle connections hence also exceptional directions. A
direction, resp.\  a point $u$ from the phase space, which is not
exceptional will be called non-exceptional.

We proceed by recalling several well known and useful (for
our purpose) results about polygonal billiards (see for example
\cite{MT}).
Recall that a flat
strip $\mathcal T$ is an invariant subset of the phase space of the
billiard flow/map such that
\begin{itemize}\item[1)] $\mathcal T$ is
contained in a finite number of floors,\item[2)] the billiard
flow/map dynamics on $\mathcal T$ is minimal in the sense that
any orbit which does not hit a corner is dense in $\mathcal T$,
\item[3)] the boundary of $\mathcal
T$ is non-empty and consists of a finite union of saddle connections.\end{itemize}

The set of the corners of $P$ is denoted by $C_P$. As usual, an $\omega$-limit set of a point $u$  is denoted by
$\omega(u)$. Let $F_P=F_P(T)$ be the set of all $u\in V_P$ for which the forward trajectory (with respect to $T$) exists.

\begin{proposition}\cite{MT}~\label{p-summary}Let $P$ be rational and $u \in F_P$. Then exactly one of the following three possibilities has to be satisfied.
\begin{itemize}\item[(i)] $u$ is periodic.
\item[(ii)] $\overline{\orb}(u)$ is a flat strip; the billiard flow/map is minimal on $\overline{\orb}(u)$.
\item[(iii)] For the flow $T_t$, $\omega(u)=R_{\proj_2(u)}$. The billiard flow/map is minimal on $R_{\proj_2(u)}$. We have
\begin{equation*}
\#(\{\proj_2(T^n(u))\colon~n\ge 0\})=2N,\end{equation*} and for
every $x\in\partial  P\setminus C_P$,
\begin{equation*}
\#\{u_0\in\omega(u)\colon~\proj_1(u_0)=x\}=N,\end{equation*} where $N=N_P$
is the least common multiple of the denominators of angles of
$P$. Moreover, in this case
\begin{equation*}\proj_2(\{u_0\in\omega(u)\colon~\proj_1(u_0)=x\})=\proj_2(\{u_0\in\omega(u)\colon~\proj_1(u_0)=x'\})\end{equation*}
whenever $x'\notin C_P$ belongs to the same side as $x$. Case (iii) holds whenever $u \in F_P$ is non-exceptional.
\end{itemize}
\end{proposition}

In general the  existence of periodic points in a polygonal billiard is still not sufficiently understood. For example, it is not known if every obtuse irrational triangle billiard has  a periodic point. But, for rational billiards the following statements holds true.

\begin{theorem}\label{dense}\cite{BGKT} If $P$ is a rational polygon, then periodic points are dense in the corresponding phase space.\end{theorem}

To deal with an irrational billiard we will use

\begin{theorem}\label{t-irrat}\cite[Theorem 4.1]{BT1}~Let $P$ be irrational and $u \in F_P$.
\begin{itemize}\item[(i)] If $\proj_2(u)$ is non-exceptional then
$\{\proj_2(T^nu)\colon~n\ge 0\}$ is infinite.
\item[(ii)] If $u$ is not periodic, but visits only a finite
number of floors then ($u$ is uniformly recurrent and) $\overline{\orb}(u)$ is
a flat strip.
\end{itemize}
\end{theorem}

For a simply connected $k$-gon $P$ we always consider counterclockwise
numbering of sides $e_1=[p_1,p_2],\dots,e_k=[p_k,p_1]$; we denote $e_i^{\circ}=(p_i,p_{i+1})$.

For $u\in V_P$, the $i$th symbolic coordinate $\sigma_i(u)$, $i\in\bbz$, is a number
$\{1,\dots,k\}$ defined by (if it exists)
\begin{equation*}\proj_1(T^iu)\in e^{\circ}_{\sigma_i}.\end{equation*}

A symbolic forward (backward, bi-infinite) itinerary of $u$ with respect to the sides of $P$ is a sequence
$$\sigma(u)=\{\sigma_i(u)\}_{i\ge 0},~(\sigma(u)=\{\sigma_i(u)\}_{i\le 0},~\sigma(u)=\{\sigma_i(u)\}_{i\in\tiny\bbz}).$$

For a sequence $\sigma=\{\sigma_i\}_{i\ge 0} \in \{1,\dots,k\}^{\tiny\bbn\cup\{0\}}$ we denote by $X(\sigma)$ the set of points from $V_P$ whose symbolic forward itinerary equals to $\sigma$.

We will repeatedly use the following result.

 \begin{theorem}\label{t1}\cite{GKT}~Let $P$ be a polygon. The following is true. \begin{itemize}\item[(i)]If $\sigma$ is periodic with period $n$ then if $n$ is even each point from $X(\sigma)$ is a periodic point of period $n$ and if $n$ is odd then there is one periodic point in $X(\sigma)$ of period $n$ and all other points from $X(\sigma)$ have period $2n$.
\item[(ii)]If $\sigma$ is non-periodic then the set $X(\sigma)$ consists of at most one point.\end{itemize}\end{theorem}

We remind the reader of the notion of
an unfolded billiard trajectory. Namely,
instead of reflecting the trajectory in a side of $P$ one may reflect $P$ in
this side and unfold the trajectory to a straight line. We leave to the reader the verification of the following fact.

\begin{proposition}\label{p:4}~Let $P$ be a polygon. For every $\delta>0$ there exists an $m=m(\delta)\in\bbn$ such that whenever $u,\tilde u\in V_P$ satisfy $$\vert\proj_2(u)-\proj_2(\tilde u)\vert>\delta$$ and the symbols $\sigma_0(u),\dots,\sigma_m(u),\sigma_0(\tilde u),\dots,\sigma_m(\tilde u)$ exist, then
\begin{equation*}(\sigma_0(u),\dots,\sigma_m(u))\neq (\sigma_0(\tilde u),\dots,\sigma_m(\tilde u)).\end{equation*}\end{proposition}

An increasing sequence $\{n(i)\}_{i\ge 0}$ of positive integers is called syndetic if the sequence $\{n(i+1)-n(i)\}_{i\ge 0}$ is bounded.
A symbolic itinerary $\sigma$ is said to be ({\it uniformly}) recurrent if for every initial word $(\sigma_0,\dots,\sigma_{m-1})$ there is a ({\it syndetic}) sequence $\{n(i)\}_{i\ge 0}$ such that
$$(\sigma_{n(i)}, \cdots, \sigma_{n(i) + m - 1}) = (\sigma_0,\cdots,\sigma_{m-1})$$ for all $i $.
For a polygon $P$ and billiard map $T\colon~V_P\to V_P$, a point $u=(x,\theta)\in F_P$ is said to be ({\it uniformly}) recurrent if for every $\varepsilon>0$ there is a ({\it syndetic}) sequence $\{n(i)\}_{i\ge 0}$ such that $$\varrho(T^{n(i)}u,u)<\varepsilon$$ for each $i$.

It is easy to see that a ({\it uniformly}) recurrent point $u$ has a ({\it uniformly}) recurrent symbolic itinerary. It is a consequence of Theorem \ref{t1} that the opposite implication also holds true.

 \begin{proposition}\cite{BT2}\label{p1}~Let $P$ be a polygon and $u\in F_P$. Then $\sigma(u)$ is (uniformly) recurrent if and only
 if $u$ is (uniformly) recurrent.\end{proposition}

For a simply connected $k$-gon $P$ we always consider counterclockwise orientation of its boundary $\partial P$. Then $[x,x']$ ($(x,x')$) denotes a closed (open) arc with outgoing endpoint $x$ and incoming endpoint $x'$.

If $P,Q$ are simply connected polygons, two sequences $\{x_n\}_{n\ge 0}\subset \partial  P$ and
$\{y_n\}_{n\ge 0}\subset \partial  Q$ have the same combinatorial order if
for each non-negative integers $k,l,m$
\begin{equation}\label{e:6}x_k\in [x_l,x_m]~\iff~y_k\in [y_l,y_m].\end{equation}

Suppose that $P$, resp.\ $Q$ is a $k$-gon with sides labeled  $e_1,e_2,\cdots,e_k$, resp.\ $f_1,f_2,\cdots,f_k$.  Thus as in Section \ref{application}
 we can think of them as identified as a common $k$-punctured sphere with marked arc-sides on the equator.
 We say that $D(P) = D(Q)$ if each closed periodic trajectory in $P$ is homotopic to a closed periodic trajectory
 in $Q$ and vice versa. From Corollary \ref{c:1} and Theorem \ref{t1}(i) we obtain

 \begin{proposition}\label{p:2}The following two conditions are equivalent.
 \begin{itemize}
 \item $D(P)=D(Q)$.
 \item $\{\sigma(u)\colon~u\in V_P\text{ is periodic}\}=\{\sigma(v)\colon~u\in V_Q\text{ is periodic}\}$
 \end{itemize}

 \end{proposition}

\begin{definition}\label{d:1}
We say that polygons  $P,Q$ are

{\it homotopically equivalent} - we write $P\equiv_{\text{hom}} Q$ - if they have the same number of sides and
\begin{itemize}

   \item[(P)] $D(P)=D(Q)$;
      \end{itemize}

  {\it order equivalent} - we write $P\equiv_{\text{ord}} Q$ - if there are points $u\in F_{P}$, $v\in F_{Q}$ for which

   \begin{itemize}

   \item[(O1)] $\overline{\{\proj_1(T^nu)\}}_{n\ge 0}=\partial  P$, $\overline{\{\proj_1(S^nv)\}}_{n\ge
   0}=\partial  Q$,
     \item[(O2)] the sequences $\{\proj_1(T^nu)\}_{n\ge 0}$, $\{\proj_1(S^nv)\}_{n\ge 0}$
  have the same combinatorial order;
  \end{itemize}

 {\it code equivalent} - we write $P\equiv_{\text{code}} Q$ - if there are  points $u\in F_{P}$, $v\in F_{Q}$ satisfying

\begin{itemize}
  \item[(C1)] $\overline{\{\proj_1(T^nu)\}}_{n\ge 0}=\partial  P$, $\overline{\{\proj_1(S^nv)\}}_{n\ge
   0}=\partial  Q$,
\item[(C2)] $\sigma(u)=\sigma(v)$.
  \end{itemize}

and {\it weakly code equivalent} - we write $P\equiv_{\text{w-code}} Q$ - if there are non-periodic points $u\in F_{P}$, $v\in F_{Q}$ satisfying

\begin{itemize}\item[(C)] $\sigma(u)=\sigma(v)$.
  \end{itemize}
   The points $u,v$ in (O1-2), (C1-2)  and (C) will be sometimes called the leaders.
 \end{definition}

The verification that these relations are reflexive, symmetric and
transitive are left to the reader.

\begin{remark}\label{r:1}The definition of weak code equivalence is new, code equivalence was
 used in \cite{BT2}. The density requirement (C1)
 was used only when proving the set $\mathcal J(e,\theta)$ from \cite[(4)]{BT2} has nonempty interior. We will see in the proof of Proposition \ref{p:1} that the requirement of density is redundant as long as $P$ is rational  and $u$ is non-exceptional (Assumption \ref{a1}).
\end{remark}

\section{General approach, the map $\mathcal R$}

In this Section we assume that polygons $P,Q$ are homotopically equivalent and $P$ is rational. By our definition $P,Q$ have the same number of sides. In this case we always consider their counterclockwise numbering $e_i=[p_i,p_{i+1}]$ for $P$, resp.\ $f_i=[q_i,q_{i+1}]$ for $Q$, where indices are taken modulo $k$. We sometimes write $e_i\sim f_i$ to emphasize the correspondence of sides $e_i,f_i$.

\begin{definition}
If $u\in F_{P}$ and $v\in F_{Q}$ satisfy $\sigma(u)=\sigma(v)$ we say that $u,v$ are related.
\end{definition}

 \begin{definition}\label{d3}Let $P$ be a polygon and $u,\tilde u\in V_P$. We say that the trajectories of $u,\tilde u$ positively intersect before their symbolic separation if
 \begin{itemize}\item[(p)] for some positive integer $\ell$,  $\sigma_{\ell}(u)\neq\sigma_{\ell}(\tilde u)$, $$\sigma_k(u)=\sigma_k(\tilde u)~\text{ whenever }k\in\{0,\dots,\ell-1\}$$  and for some $k_0\in\{0,\dots,\ell-1\}$, the segments with endpoints $$\proj_1(T^{k_0}u),\proj_1(T^{k_0+1}u)\text{ and }\proj_1(T^{k_0}\tilde u),\proj_1(T^{k_0+1}\tilde u)$$ intersect; similarly, the trajectories of $u,\tilde u$ negatively intersect before their symbolic separation if
 \item[(n)] for some negative integer $\ell$,  $\sigma_{\ell}(u)\neq\sigma_{\ell}(\tilde u)$, $$\sigma_k(u)=\sigma_k(\tilde u)~\text{ whenever }k\in\{\ell+1,\dots,0\}$$  and for some $k_0\in\{\ell,\dots,-1\}$, the segments with endpoints $$\proj_1(T^{k_0}u),\proj_1(T^{k_0+1}u)\text{ and }\proj_1(T^{k_0}\tilde u),\proj_1(T^{k_0+1}\tilde u)$$ intersect.\end{itemize}
 We shortly say that trajectories of $u,\tilde u$ intersect before their symbolic separation if either (p) or (n) is fulfilled.

The pairs $u,\tilde u\in V_P$ and $v,\tilde v\in V_Q$ have the same type of forward symbolic separation (with parameter $k$) if for some nonnegative $k$
\begin{align*}\label{a:1}&\sigma_i(u)=\sigma_i(\tilde u)=\sigma_i(v)=\sigma_i(\tilde v), ~i=0,\dots,k-1,\\
  &\sigma_k(u)=\sigma_k(v)\neq\sigma_k(\tilde u)=\sigma_k(\tilde v)\notag.
 \end{align*}
The same type of backward symbolic separation of pairs $u,\tilde u\in V_P$ and $v,\tilde v\in V_Q$ (with some non-positive parameter) is defined analogously.

\end{definition}

\begin{proposition}\label{p5}
Let $P,Q$ be homotopically equivalent, assume that $u,v$, resp.\ $u',v'$ are related.
\begin{itemize}
\item[(i)] The pairs $u,u'$ and $v,v'$  have the same type of symbolic separation. This is also true for the pairs $T^mu,T^nu$ and $S^mv,S^nv$ and each nonnegative $m,n$.
\item[(ii)]If for some nonnegative $m,n$, $\proj_2(T^mu)=\proj_2(T^nu)$ then trajectories of $S^mv,S^nv$ cannot intersect before their symbolic separation.\end{itemize}\end{proposition}
\begin{proof}The part (i) follows from our definitions of the same separation and related points. For (ii) see \cite[Proposition 4.3]{BT2}.\end{proof}

\noindent{\bf The map $\mathcal R$.}

Let $P,Q$ be homotopically equivalent, $P$ rational. In what follows we define a map
$$\mathcal R\colon~F_{P,\infty}\to \bigcup_f\left(f\times (-\frac{\pi}{2},\frac{\pi}{2})\right),$$
where $F_{P,\infty}$ is the set of points with bi-infinite orbit which are not parallel to a saddle connection, more formally the set of all non-exceptional points from $F_P(T)\cap F_P(T^{-1})$; the union is taken over all sides of $Q$.

Choose $u=(x,\theta)\in F_{P,\infty}$ with $x\in e=e_{\sigma_0(u)}$. We know from Theorem \ref{t1}(i) that $\sigma(u)$ is non-periodic. By Theorem \ref{dense}
there exists a sequence
$(u_n)_{n=1}^{\infty}$ of periodic points which converge to the point $u$ with respect to $\varrho$ from (\ref{rho}), i.e., $\lim_n\varrho(u_n,u)=0$. Fix $u$ and the approximating sequence
$(u_n)$.
  \begin{proposition}\label{c:2}If $v_n$ is related to $u_n$ then the limit
$ v = \lim_{n \to \infty} v_n $ exists and
$\vartheta=\lim_n\proj_2(v_n)\in (-\frac{\pi}{2},\frac{\pi}{2})$. Furthermore
the point $v$ depends only on $u$ and not on the approximating sequences $u_n,v_n$.
\end{proposition}
 \begin{proof} On the one hand, since $v_n$ is related to $u_n$, by Proposition \ref{p5} the pairs $u_n,u_m$ and $v_m,v_n$ have the same type of forward, resp.\ backward symbolic separation with parameter $k(m,n)$, resp.\ $\ell(m,n)$. We assume that $u$ has a bi-infinite symbolic itinerary and $\lim_n\varrho(u_n,u)=0$. It implies that
 $\lim_{m,n\to\infty}k(m,n)=\lim_{m,n\to\infty}-\ell(m,n)=\infty$. On the other hand, if the limit $\lim_n\proj_2(v_n)$ does not exist, for some positive $\delta$ there are arbitrarily large $m,n$ such that $\vert\proj_2(v_m)-\proj_2(v_n)\vert>\delta$. Then by Proposition \ref{p:4} applied to the maps $S,S^{-1}$ (for $Q$), $$\max\{k(m,n),-\ell(m,n)\}\le m(\delta).$$ This shows that the limiting
 angle $\vartheta\in [-\frac{\pi}{2},\frac{\pi}{2}]$  exists.

 To the contrary, let for instance $\vartheta=-\frac{\pi}{2}$. By Theorem \ref{dense} we can consider a periodic point $u'$ satisfying \begin{equation*}\proj_1(u')\in e^{\circ},~x<\proj_1(u'),~\proj_2(u')\in (-\frac{\pi}{2},\theta).\end{equation*} Now let $v'$ be a periodic point in $Q$ related to $u'$. Clearly $\proj_2(v')\in (-\frac{\pi}{2},\frac{\pi}{2})$ and for any sufficiently large $n$
\begin{equation*}\label{e:3}\proj_2(u_n)>\proj_2(u')\ \ \ \ \&\ \ \ \ \proj_2(v_n)<\proj_2(v'),\end{equation*}
hence the pairs $u_n,u'$, $v_n,v'$ do not have the same symbolic separation, a contradiction with Proposition \ref{p5}. The equality $\vartheta=\frac{\pi}{2}$ can be disproved analogously.

 If the limit $\lim_{n\to\infty}\proj_1(v_n)$ does not exist, there are distinct points $\alpha,\beta\in f=f_{\sigma_0(u)}$ such that (w.l.o.g.)
\begin{equation*}\lim_{n\to\infty}\proj_1(v_{2n+1})=\alpha,~\lim_{n\to\infty}\proj_1(v_{2n})=\beta\end{equation*}
and $[\alpha,\beta]\subset f=f_{\sigma_0(u)}$. Then for each foot point $y$ in $(\alpha,\beta)$ necessarily $v=(y,\vartheta)\in F_Q$ and $\sigma(v)=\sigma(u)$, what is by Theorem \ref{t1} impossible for non-periodic $\sigma(u)$. Thus, $\lim_nv_n=v=(y,\vartheta)$ exists with
$\vartheta \in (\frac{-\pi}{2}, \frac{\pi}{2})$.

Next suppose that $\tilde{v}_n$ is another point  related to $u_n$. Consider the sequence
$u'_n$ defined by $u'_{2n} = u_n$ and $u'_{2n+1} = u_n$ and the sequence
 of related points $v'_n$ defined by $v'_{2n} = v_n$ while $v'_{2n+1} = \tilde{v}_n$.
 By the above the limit $\lim_{n \to \infty} v'_n$ exists, thus by passing to subsequences
 $\lim_{n \to \infty} v_n = \lim_{n \to \infty} \tilde{v}_n$.

Suppose now that we have two sequences of periodic points $u_n \to u$ and $\tilde{u}_n \to u$
with there associated sequences of related point $v_n$ and $\tilde{v}_n$.
Create a new sequence $u'_n$ by setting $u'_{2n} = u_n$ and $u'_{2n+1} = \tilde{u}_n$
and similarly create a sequence $v'_n$.
Again applying the above the limit $\lim_{n \to \infty} v'_n$ exists and $\lim_{n \to \infty} v_n = \lim_{n \to \infty} \tilde{v_n}$.
\end{proof}
Thus the map  $$\mathcal R(u)=v$$ is well defined.

\begin{lemma}\label{l:6}For $u\in F_{P,\infty}$ and each $n\in\bbz$ for which the element $S^n\mathcal R(u)$ is defined well,
$$ \mathcal R(T^nu)=S^n\mathcal R(u).$$
In particular, if $\mathcal R(u)\in F_Q$ then $u$ and $\mathcal R(u)$ are related.
\end{lemma}
\begin{proof}It follows from the definition of $\mathcal R$ and the continuity of the billiard maps $T,S$.\end{proof}

A $u\in V_P(T)$ is an element of a saddle connection if there are distinct nonnegative $k$ and non-positive $\ell$ such that $\proj_1(T^ku),\proj_1(T^{\ell}u)\in C_P$. Denote
$$G_P=\{u\in V_P\colon~u\text{ is an element of a saddle connection}\}.$$
Clearly $G_P(T)=G_P(T^{-1})$ and since each saddle connection contains finitely many elements and there are countably many saddle connections, the set $G_P$ is countable.

\begin{lemma}\label{l:7}(i) For each $v\in \mathcal R(F_{P,\infty})\setminus G_Q$, $\card~\mathcal R^{-1}(v)=1$.\\ (ii) For each $v\in \mathcal R(F_{P,\infty})\cap G_Q$, $\card~\mathcal R^{-1}(v)\le 2$.
\end{lemma}
\begin{proof}(i) Suppose that  $\card~\mathcal R^{-1}(v)>1$ for some
$v\in \mathcal R(F_{P,\infty})$. Choose distinct points $u,u'\in \mathcal R^{-1}(v)$ and using Theorem \ref{dense} consider sequences $(u_n)_{n=1}^{\infty}$, $(u'_n)_{n=1}^{\infty}$ of periodic points satisfying \begin{equation}\label{e:1}\lim_n\varrho(u_n,u)=\lim_n\varrho(u'_n,u')=0.\end{equation}
Since $u,u' \in F_{P,\infty}$ this implies that $\sigma(u_n) \to \sigma(u)$ and $\sigma(u'_n) \to \sigma(u')$.
Let $v_n$ be related to  $u_n$, and $v'_n$ to $u'_n$. Proposition \ref{c:2} implies that $\lim_nv_n=\lim_nv'_n=v$.
Since $u,u' \in F_{P,\infty}$, they are
not periodic and have bi-infinite orbits, thus Theorem
\ref{t1} implies that their forward codes $\sigma(u),\sigma(u')$ differ at some time, let $k \ge 0$ be the first such time.  Similarly there backwards
codes differ at some first time $\ell \le 0$.

Suppose first that $\proj_1(v) \not \in C_Q$, then $\ell < 0 < k$.
 For each sufficiently large $n$ the pairs $u,u'$, $u_n,u'_n$ and $v_n,v'_n$ have the same type of forward, resp.\ backward symbolic separation, i.e.\
 for each sufficiently large $n$  the bi-infinite codes of $u_n$ and $u'_n$ agree for all
 $\ell < i < k$ and disagree at times $\ell$ and $k$ (see Proposition \ref{p5}).  Since $\lim_nv_n=\lim_nv'_n=v$, this implies that $S^kv,S^{\ell}v\in C_Q$ and thus $v\in G_Q$.

We now turn to the case when $\proj_1(v) \in C_Q$ and thus $k=\ell=0$. Then $\proj_1(Sv)\notin C_Q$ since by assumption $v \not \in G_Q$. Using Lemma \ref{l:6}, one can apply the previous reasoning to the points $Sv$, $Tu$, $Tu'$, $Tu_n$, $Tu'_n$, $Sv_n$ and $Sv'_n$ with respect to billiard maps $T^{-1},S^{-1}$ to show that $Sv\in G_Q(S^{-1})$ hence also $v\in G_Q(S)$, a contradiction.

(ii) Assume that for some $v\in \mathcal R(F_{P,\infty})\cap G_Q$, $\card~\mathcal R^{-1}(v)> 2$. If $v \not \in C_Q$ then the footpoint $\proj_1(u)$
of each $u \in \mathcal R^{-1}(v)$ belongs to one and the same side of $P$.
If $v \in C_Q$ then
any point in $\mathcal R^{-1}(v)$ must have
footpoint in one of two consecutive sides $e_i$ or $e_{i+1}$, where $v \in f_i \cap f_{i+1}$ and $e_i \sim f_i,\ e_{i+1} \sim f_{i+1}$.
In either case there are two points $u,u'\in R^{-1}(v)$ whose footpoints belong to the same side of $P$ and there exists the least positive integer $k$ such that $\sigma_k(u)\neq \sigma_k(u')$. Clearly,
$\proj_1(\mathcal R(T^{k}u))=\proj_1(\mathcal R(T^{k}u'))\in C_Q$. Let $\ell\in\{0,\dots,k-1\}$ be the largest value satisfying $\proj_1(\mathcal R(T^{\ell}u))\in C_Q$ if it exists or $\ell=0$ otherwise.

Using Definition \ref{d3}, a sequence of periodic points $(u_n)_{n=1}^{\infty}$ fulfilling (\ref{e:1}) will be called positive if the trajectories of no pair $u_m,u_n$ positively intersect before their symbolic separation. Consider positive sequences $(u_n)_{n=1}^{\infty}$, $(u'_n)_{n=1}^{\infty}$ of periodic points satisfying (\ref{e:1}) and let  $(v_n)_{n=1}^{\infty}$, $(v'_n)_{n=1}^{\infty}$ be sequences of related periodic points. By our definition of the numbers $k,\ell$ for one of those sequences - w.l.o.g. we can assume that it is $(v_n)$ - for some large $n$ the segments with endpoints
\begin{equation*}\proj_1(S^{\ell}v_n),~\proj_1(S^kv_n)\text{ and }\proj_1(\mathcal R(T^{\ell}u)),~\proj_1(\mathcal R(T^{k}u))\end{equation*}
intersect. But then also for that $n$ and a sufficiently large $m$,  $v_m,v_n$ positively intersect before their symbolic separation. It is not possible by our choice of $(u_n)$ and Proposition \ref{p5}, a contradiction.
\end{proof}

In a rational polygon we say that a point $u$ is
generic if it is non-exceptional, has bi-infinite orbit
and the billiard  map  restricted to the skeleton $V_P \cap R_{\proj_2(u)}$ of the
invariant surface  $R_{\proj_2(u)}$
has a single invariant measure (this measure is then automatically the
measure $\mu$).

\begin{lemma}\label{c:3}Let $P$,$Q$ be homotopically equivalent polygons with $P$ rational. There are related points $u\in F_P$ and $v\in F_Q$ with $u$ generic.  \end{lemma}
\begin{proof}Recall that by our definition, the map $\mathcal R$ has its domain equal to the set $F_{P,\infty}$ of all non-exceptional points with bi-infinite orbit. In particular, it contains an uncountable set (of the full natural invariant measure) of generic points \cite{MT}. Since by Lemma \ref{l:7}(ii) the set $\mathcal R^{-1}(G_Q)$ is countable, we can take a generic point $u\in F_{P,\infty}\setminus\mathcal R^{-1}(G_Q)$, let $v=\mathcal R(u)$. We are done if $v\in F_Q$, since then by Lemma \ref{l:6}, $\sigma(u)=\sigma(v)$, i.e., $u$ and $v$ are related. If for some non-negative $k$, $\proj_1(S^kv)\in C_Q$,  we choose $u'=T^{k+1}u$ and $v'=\mathcal R(T^{k+1}u)$. Then from the fact that $v\notin G_Q$ and Lemma \ref{l:6} we get $v'\in F_Q$ and $u',v'$ are related. Moreover, the point $u'$ is (as an image of generic $u$) generic.\end{proof}

To state Proposition \ref{p:1} we need the following

\begin{assumption}\label{a1}Let $P$,$Q$ be polygons with $P$ rational, and let $u\in F_P$ and  $v\in F_Q$ be related points with $u$ non-exceptional.\end{assumption}
\begin{remark}As we have mentioned in Remark \ref{r:1} we do not assume that $\overline{\{\proj_1(S^nv)\}}_{n\ge 0}$ has a nonempty interior.\end{remark}
For a side $e$ of $P$ and $\theta\in (-\frac{\pi}{2},\frac{\pi}{2})$ we put

\begin{equation}\label{e16}I(u,e,\theta)=\{n\in\bbn\cup\{0\}\colon~\proj_1(T^nu)\in e,~\proj_2(T^nu)=\theta\}.\end{equation}

Throughout the section let $u_n=T^nu$, $x_n=\proj_1(u_n)$, $v_n=S^nv$, $y_n=\proj_1(v_n)$, $\alpha_n=\proj_2(u_n)$, $\beta_n=\proj_2(v_n)$. The set $I(u,e,\theta)$ defined for a side $e=e_i$ in (\ref{e16})
is nonempty only for $\theta$'s from the set $$\{\proj_2(T^nu): n \ge 0 \}$$
which  is
finite by Proposition \ref{p-summary} for any $u$, for rational $P$. In what follows we fix such $e$ and $\theta$.

From Proposition \ref{p-summary} $u$ is uniformly recurrent and thus from
Proposition \ref{p1}  $v$ is uniformly recurrent as well.

Obviously the set
\begin{equation*}\label{e12}\mathcal J=\mathcal J(e,\theta)=\overline{\{y_n\colon~n\in I(u,e,\theta)\}}\end{equation*}
is a perfect subset of a side $f \sim e$. The counterclockwise orientation of $\partial Q$ induces the linear ordering of $f$ and we can consider two elements $\min \mathcal J,\max \mathcal J\in f$.

Define a function $g\colon~\{y_n\}_{n\in I(u,e,\theta)}\to (-\frac{\pi}{2},\frac{\pi}{2})$ by $g(y_n)=\beta_n$.

\begin{proposition}\label{p:1}Under Assumption \ref{a1} the following is true.
\begin{itemize}\item[(i)] The function $g$ can be extended continuously to the map $$G\colon~\mathcal J\to [-\frac{\pi}{2},\frac{\pi}{2}].$$ Moreover, $G(y)\in (-\frac{\pi}{2},\frac{\pi}{2})$ for each $$y\in \mathcal J\setminus \{\min \mathcal J,\max \mathcal J\}.$$

\item[(ii)]The sequences $\{x_n\}_{n\in I(u,e,\theta)}\subset e$ and $\{y_n\}_{n\in I(u,e,\theta)}\subset f$ have the same combinatorial order due to (\ref{e:6}).
\item[(iii)]The set $\mathcal J$ is an interval.
\item[(iv)]The function $G$ is constant on $\mathcal J$.
\item[(v)]The set of directions $$\{\proj_2(S^nv)\colon~n\ge 0\}$$
along the trajectory of $v$ is finite.
\end{itemize}
\end{proposition}

\begin{proof}For the proofs of (i)-(ii) and (iv)-(v) - see \cite{BT2}. Assume that the conclusions (i),(ii) are true and $\mathcal J$ is not an interval. Since $P$ is rational, we know from  Proposition \ref{p-summary} that the sequence $\{x_n\}_{n\in I(u,e,\theta)}$ is dense in $e$. Using (ii) we can consider two sequences $\{n(2k)\}_k$ and $\{n(2k+1)\}_k$ such that in $e$, resp.\ $f$
\begin{equation}\label{e:4}x_{n(2k)}\nearrow \tilde x=x'\swarrow x_{n(2k+1)},\text{ resp.\ }y_{n(2k)}\nearrow \tilde y<y'\swarrow y_{n(2k+1)}.
\end{equation}
Now, since $u$ is non-exceptional, the point $u'=(x',\theta)$ is also non-exceptional and either $u'\in F_P(T)$ or  $u'\in F_P(T^{-1})$. Without loss of generality assume the first possibility.

\noindent {\bf I.} Using (i) we see that $\tilde v=(\tilde y,G(\tilde y))$, $v'=(y',G(y'))$ are defined well. Let us show that $G(\tilde y)=G(y')$. If not, for some positive
$\delta$ and each sufficiently large $k$ we get (as above, $v_n=(y_n,\beta _n)$)
\begin{equation*}\vert \beta_{n(2k)}-\beta_{n(2k+1)}\vert>\delta\end{equation*}
and for an $m(\delta)$ guaranteed by Proposition \ref{p:4} and each large $k$
\begin{equation*}(\sigma_0(v_{n(2k)}),\dots,\sigma_m(v_{n(2k)}))\neq (\sigma_0(v_{n(2k+1)}),\dots,\sigma_m(v_{n(2k+1)})),\end{equation*}
what contradicts our choice of related \begin{equation*}\label{e:5}u_{n(k)}=(x_{n(k)},\theta),~v_{n(k)}=(y_{n(k)},\beta_{n(k)}).\end{equation*}

\noindent {\bf II.} Choose arbitrary $y^*\in (\tilde y,y')$ and show that for $v^*=(y^*,G(y'))$, $v^*\in F_Q$ and $\sigma(v^*)=\sigma(u')$. If to the contrary either $\proj_1(S^mv^*)\in C_Q$ or $\sigma_m(v^*)\neq\sigma_m(u')$ then by {\bf I.} also $\sigma_m(v_{n(2k)})\neq \sigma_m(v_{n(2k+1)})$ for each sufficiently large $k$, what contradicts our choice of $u_n$ and related $v_n$. But then, $\sigma((y^*,G(y')))=\sigma((y^{**},G(y')))$ for any distinct $y^{*},y^{**}\in (\tilde y,y')$. This is impossible by Theorem \ref{t1}(ii).
\end{proof}

\begin{theorem}\label{abc}
Let $P,Q$ be weakly code equivalent with leaders $u,v$ with $u$ non-exceptional and $P$ rational.
Then
$v$ is non-exceptional and
 $Q$ is rational.
\end{theorem}

\begin{proof} Applying Theorm \ref{p:1}(v) we can use literally the same proof as in \cite[Theorem 5.3]{BT2}.\end{proof}

\begin{corollary}\label{t8}Let $P$, $Q$, with $P$ rational be homotopically equivalent. Then
\begin{itemize}\item[(i)]I if $u,v$ are related and $u$ is non-exceptional then also $v$ is non-exceptional.
\item[(ii)] $Q$ is rational.
\end{itemize}
\end{corollary}

\begin{proof} This follows from Theorem \ref{abc} since we have seen in Lemma \ref{c:3} that there are related points $u\in F_P$ and $v\in F_Q$ with $u$ non-exceptional.\end{proof}

\begin{corollary}\label{t7}Let $P$, $Q$ with $P$ rational be homotopically equivalent. Then ($Q$ is rational and) there are generic related points.
\end{corollary}
\begin{proof}As a result of our construction of the map $\mathcal R$ and Lemma \ref{c:3} we can consider related points $u,v$ with $u$ generic. Then Proposition \ref{p:1} applies and the map $G\equiv\vartheta$ is constant on the interval $\mathcal J$. Since $G$ is constant, there is a countable subset $\mathcal J_0$ of $\mathcal J$ such that each point
from $(\mathcal J\setminus \mathcal J_0)\times\vartheta$ has a bi-infinite
trajectory (either the forward or backward trajectory starting from
any point of $\mathcal J_0\times\vartheta$ finishes in a corner of $Q$).

The analogous argument says that $(x,\theta)$ is generic (on the same surface $R_{\theta}$) with bi-infinite trajectory for each $x\in e\setminus e_0$, where $e_0$ is countable - see the definition of the set $\mathcal J$. Using Proposition \ref{p:1}(ii) one can consider a sequence $\{n(k)\}_k$ such that in $e$, resp.\ $f$
\begin{equation*}x_{n(k)}\nearrow x'\in e\setminus e_0,\text{ resp.\ }y_{n(k)}\nearrow y'\in\mathcal J\setminus\mathcal J_0
\end{equation*}
and the points $u'=(x',\theta)$, $v'=(y',\vartheta)$ are related. Then $v'$ is non-exceptional by Corollary \ref{t8}(i) and using the map $\mathcal R$ restricted to the skeleton $K_P$, by Lemma \ref{l:6} there is a unique invariant measure $\nu=\mathcal R^*\mu$ on the skeleton $K_Q$, so also the point $v'$ is generic.\end{proof}

\section{The Results}

We define one more equivalence relation.  Suppose $P,Q$ are polygons with
$P$ rational.  Remembering that the notion $N_P$ was defined in Proposition \ref{p-summary}, we  define $P \equiv_{\text{sim}} Q$ by
\begin{enumerate}
\item{}  if $N_P\ge 3$, $Q$ is similar to $P$;
\item{}  if $N_P = 2$, $Q$ is affinely similar to $P$.
\end{enumerate}

\begin{theorem}\cite{BT2}\label{converse}
Suppose $P,Q$ are code equivalent with leaders $u,v$; $P$ rational, $u$ generic.  Then $P,Q$ are order equivalent with leaders $u,v$.
\end{theorem}

\begin{theorem}\cite{BT1}\label{o-similar} Let $P,Q$ be order equivalent with leaders $u,v$; $P$ rational, $u$ generic. Then $Q \equiv_{\text{sim}} P$.
\end{theorem}

\begin{theorem}\label{compare} Let $P,Q$ be polygons, $P$ rational, then
the following are equivalent.
\begin{enumerate}
\item
$P\equiv_{\text{hom}} Q$
\item $P\equiv_{\text{ord}} Q$ with leaders $u,v$;  $u$ generic
\item $P\equiv_{\text{code}} Q$ with leaders $u,v$; $u$ generic
\item $P\equiv_{\text{w-code}} Q$ with leaders $u,v$; $u$ generic
\item $P\equiv_{\text{sim}} Q$.
\end{enumerate}
\end{theorem}
\begin{proof}
Clearly (5) implies (1), (2), (3) and (4).
Clearly (3) implies (4).
Corollary \ref{t7}  shows that (1) implies (4),
 while Theorem \ref{converse}
shows that (4) implies (2).
Finally Theorem \ref{o-similar} shows that (2) implies (5),
\end{proof}

\begin{corollary}\label{main} Let $P,Q$, $P$ rational be homotopically equivalent.
Then $Q$ is rational such that
\begin{enumerate}
\item{}  if $N(P)\ge 3$, $Q$ is similar to $P$;
\item{}  if $N(P) = 2$, $Q$ is affinely similar to $P$.
\end{enumerate}
\end{corollary}

\begin{example}
Consider an arbitrary  figure $L$ billiard table. Code the sides as in Figure 1.
Then for any strictly positive integer $n$ there is a periodic billiard orbit
with code $(lr)^n (t) (lr)^n (b)$.  These orbits are pictures for $n=1$ and $n=3$ in a partial
unfolding in Figure 1. Thus it does not suffice to assume that $D(P) \cap D(Q)$ is infinite to
conclude that $P\equiv_{\text{sim}} Q$.
\end{example}

\begin{figure}[t]
\label{thefig}

\centering
\begin{tikzpicture} [scale=0.75]
\draw [] (0,0) -- (0,3) -- (3,3) -- (3,1) -- (5,1) -- (5,0) --(0,0);
   \draw (-0.5,1) node {$\ell$};
\draw  (5.5,0.5) node {$r$};
\draw (2.5,-0.5) node {$b$};
\draw (4,1.5) node {$t$};
\end{tikzpicture}

\vspace{.2in}
\centering
\begin{tikzpicture} [scale=0.45]
\draw [] (0,0) -- (0,3) -- (3,3) -- (3,1) -- (5,1) -- (5,0) --(0,0);
\draw [] (5,0) -- (5,1) -- (7,1) -- (7,3) -- (10,3) -- (10,0) --(5,0);
\draw [] (10,0) -- (10,3) -- (13,3) -- (13,1) -- (15,1) -- (15,0) --(10,0);
\draw [] (15,0) -- (15,1) -- (17,1) -- (17,3) -- (20,3) -- (20,0) --(15,0);
\draw [] (20,0) -- (20,3) -- (23,3) -- (23,1) -- (25,1) -- (25,0) --(20,0);
\draw [] (25,0) -- (25,1) -- (27,1) -- (27,3) -- (30,3) -- (30,0) --(25,0);

\draw [] (0,0.6) -- (4,1) -- (10,0.4) -- (6,0 ) -- (0,0.6);
\draw [dashed] (0,0.55) -- (13.5,1) -- (30,0.45) -- (17.5,0) -- (0,0.55);

\end{tikzpicture}

    \caption{Two periodic orbits drawn in unfolding}

\end{figure}
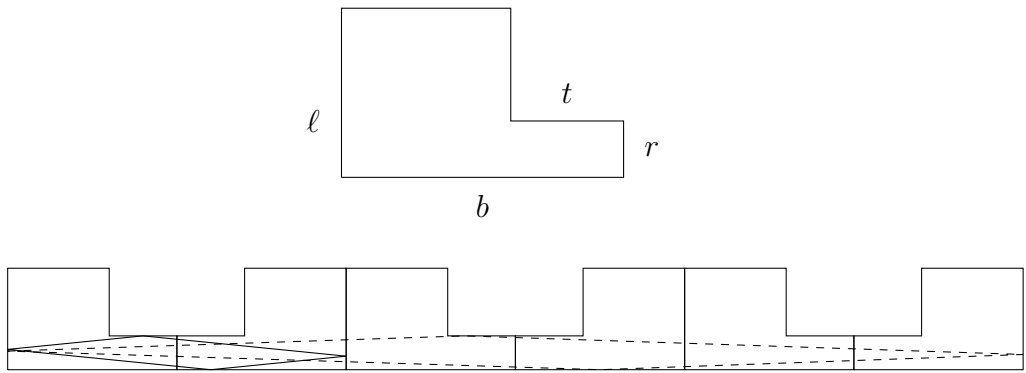

\subsection*{Acknowledgements}
The first author was partially supported by the Grant Agency of the Czech Republic contract number
201/09/0854. He also gratefully acknowledges the support of the MYES of the
Czech Republic via the contract MSM 6840770010. The second author was partially supported by ANR Perturbations.

\end{document}